\documentclass[11pt]{amsart}
\usepackage[mathscr]{eucal}
\usepackage[dvipdfm]{graphicx}
\usepackage{amssymb, amscd, color}

 \pdfoutput=1
\theoremstyle{plain}
\newtheorem{theorem}{Theorem}[section]
\newtheorem{corollary}{Corollary}[section]
\newtheorem{lemma}{Lemma}[section]
\newtheorem{proposition}{Proposition}[section]

\newtheorem{definition}{Definition}[section]

\theoremstyle{remark}
  \newtheorem{remark}{Remark}[section]
  \newtheorem{example}{Example}[section]

\numberwithin{equation}{section}
\def\Vec#1{\mbox{\boldmath $#1$}}
\newcommand{\C}{\mathbb C}
\newcommand{\diag}{\operatorname{diag}}

\newcommand{\bbar}{\left(  \begin{array}}
\newcommand{\ebar}{\end{array} \right)}
\newcommand{\bdm}{\begin{displaymath}}
\newcommand{\edm}{\end{displaymath}}
\newcommand{\beq}{\begin{equation}}
\newcommand{\beqa}{\begin{eqnarray}}
\newcommand{\beqas}{\begin{eqnarray*}}
\newcommand{\eeq}{\end{equation}}
\newcommand{\eeqa}{\end{eqnarray}}
\newcommand{\eeqas}{\end{eqnarray*}}
\newcommand{\dd}{\textup{d}}
\newcommand{\real}{\mathbb R}
\newcommand{\SSS}{{\mathbb S}}
\newcommand{\HHH}{{\mathbb H}}
\newcommand{\Ad}{\textup{Ad}}

\begin{document}
\title[CGC surfaces in $\mathbb{S}^3$]
{Constant Gaussian curvature surfaces in the 3-sphere via loop groups}
\author[D.~Brander]{David Brander}
\address{
Institut for Matematik og Computer Science,
Matematiktorvet, bygning 303B,
Technical University of Denmark,
DK-2800 Kgs. Lyngby, Denmark} 
\email{D.Brander@mat.dtu.dk}
\thanks{The first named author is partially supported by 
FNU grant \emph{Symmetry Techniques in Differential Geometry}}
\author[J.~Inoguchi]{Jun-ichi Inoguchi}
\address{Department of Mathematical Sciences, 
Faculty of Science,
Yamagata University, 
Yamagata, 990--8560, Japan}
\email{inoguchi@sci.kj.yamagata-u.ac.jp}
\thanks{The second named author is partially supported by Kakenhi 
21546067, 24540063}
\author[S.-P.~Kobayashi]{Shimpei Kobayashi}
\address{
Graduate School of Science and Technology,
Hirosaki University,
Hirosaki, 036-8561, Japan}
\email{shimpei@cc.hirosaki-u.ac.jp}
\thanks{The third named author is partially supported by Kakenhi 
23740042}
\date{\today}

\begin{abstract}
In this paper we study constant positive  Gauss curvature $K$ surfaces in the 
$3$-sphere $\mathbb{S}^3$ with $0<K<1$ as well as constant negative 
curvature surfaces. We show that the so-called \emph{normal} Gauss map
for a surface in $\mathbb{S}^3$ with Gauss curvature $K<1$ is Lorentz harmonic
with respect to the metric induced by the second fundamental form 
if and only if $K$ is constant. 
 We give a uniform loop group formulation for all such
surfaces with $K\neq 0$, and use the generalized d'Alembert method to construct examples.
This representation gives a natural correspondence between such surfaces with $K<0$
and those with $0<K<1$. 
\end{abstract}

\maketitle

\section*{Introduction}

The study of isometric immersions from space 
forms into space forms is a classical and important 
problem of differential geometry. 
This subject has its origin in realizability of the hyperbolic plane geometry in 
Euclidean 3-space $\mathbb{E}^3$. As is well known, 
Hilbert proved the nonexistence of isometric immersions of the
hyperbolic plane into $\mathbb{E}^3$ \cite{Hilbert}.
Analogous results hold for surfaces in the $3$-sphere 
$\mathbb{S}^3$ and hyperbolic $3$-space $\HHH^3$ as follows:
\begin{theorem}[\cite{Spivak}]
There is no complete surface 
in $\mathbb{S}^3$ or $\mathbb{H}^3$ with constant negative curvature $K< 0$ and 
constant negative extrinsic curvature.
\end{theorem}
Due to the complicated structure (nonlinearity) of the integrability 
condition (Gauss-Codazzi-Ricci equations) of isometric immersions 
between space forms,  
in the past decades many results on 
non-existence, rather than the construction 
of explicit examples, have been obtained. For this direction we refer the reader to 
a survey article \cite{Borishenko}.

Another reason for the focus on non-existence may be the presence of singularities.
Surfaces in $\mathbb{S}^3$ with constant Gauss curvature $K<1$ always have 
singularities, excepting the flat case $K=0$ (in fact there exist infinitely many 
flat tori in $\mathbb{S}^3$ \cite{Kitagawa}).
Recently, however, there has been some movement to broaden the class of surfaces to include
those with singularities, and a number of interesting studies 
of the geometry of these:  see \textit{e.g.} \cite{SUY}.

On the other hand, one can see that under the asymptotic Chebyshev net parametrization,
the Gauss-Codazzi equation of surfaces in $\mathbb{S}^3$ with constant curvature 
$K<1$ ($K\not=0$) are reduced to the sine-Gordon equation. 
The sine-Gordon equation also arises as the Gauss-Codazzi equation of pseudospherical surfaces
in $\mathbb{E}^3$ (surfaces of constant negative curvature) and is associated to harmonic maps
from a Lorentz surface into the $2$-sphere.

By virtue of loop group techniques an infinite dimensional d'Alembert type 
representation for solutions is available for surfaces associated to Lorentz harmonic
maps. More precisely, 
all solutions are given in terms of two functions, each of one variable only. 
This type of construction method can be traced back to a work by 
Krichever \cite{Krichever}. An example of an application of this method is the
solution, in \cite{BS1}, of the geometric Cauchy problem for pseudospherical surfaces in $\mathbb{E}^3$
as well as for timelike constant mean curvature (CMC) surfaces in Lorentz-Minkowski 3-space $\mathbb{L}^3$.  The key ingredient 
is the generalized d'Alembert 
representation for Lorentz harmonic maps of Lorentz surfaces into semi-Riemannian symmetric spaces. See also \cite{D} and the references therein for more examples.  One can expect that the 
approach can be adapted to other classes of isometric immersion problems.

These observations motivate us to establish 
a loop group method (generalized d'Alembert formula) for surfaces in $\mathbb{S}^3$ of 
constant curvature $K<1$.
We shall in fact give such a solution that covers all such surfaces with $K\neq0$.
The key point is to discover \emph{which Gauss map} (there are several definitions for 
surfaces in $\SSS^3$) is the right one to make the connection with harmonic maps.

\subsection{Outline of this article}
This paper is organized as follows. 
After prerequisite knowledge in Sections 1--2, we will give a loop group formulation 
for surfaces in $\mathbb{S}^3$ of constant curvature $K<1$, ($K\not=0$) 
in Section 3. In particular we will show that the Lorentz harmonicity (with respect to 
the conformal structure determined by the second fundamental form) of the \emph{normal Gauss map}
of a surface with curvature $K<1$ is equivalent to the constancy of $K$. 
The normal Gauss map is the left translation, to the Lie algebra $\mathfrak{su}(2)$, of
the unit normal $n$ to the immersion $f$ into $\SSS^3 = SU(2)$: namely $\nu = f^{-1}n$.

The harmonicity of the normal Gauss map enables 
us to construct constant curvature surfaces in terms of  
Lorentz harmonic maps. We  establish a loop group theoretic d'Alembert representation for 
surfaces in $\mathbb{S}^3$ with constant curvature  $K<1$, ($K\not=0$). 
In Section \ref{limitingsection} we  give 
a relation between the surfaces in $\mathbb{S}^3$ and pseudospherical surfaces in $\mathbb{E}^3$,
and show how the well known Sym formula for the latter surfaces arises naturally from our
construction. 
Finally, we  give a detailed analysis of the limiting procedure $K\to 0$. 
The paper ends with some explicit examples constructed by our method. 

\begin{figure}[ht]
\centering
$
\begin{array}{ccc}
\includegraphics[height=45mm]{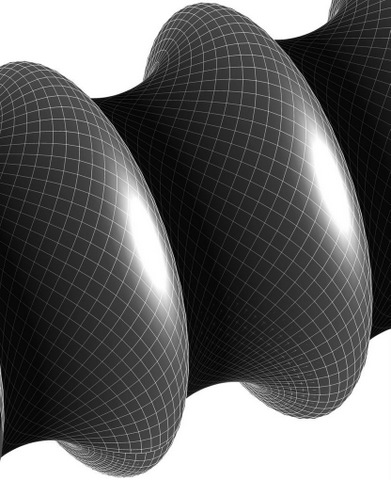} \quad  &  \quad 
\includegraphics[height=45mm]{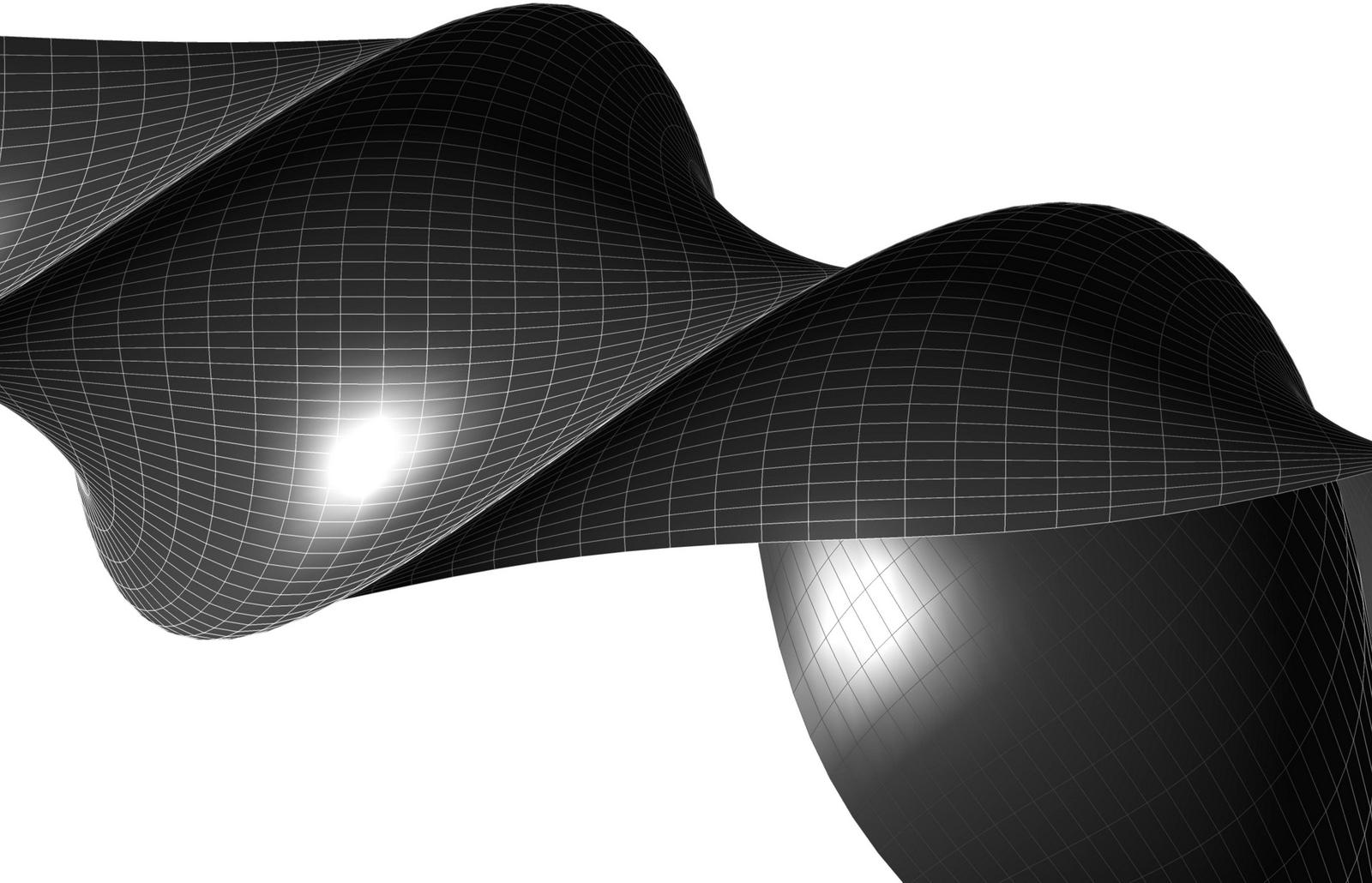}   
 \vspace{2ex} \\
\mu=1, ~ K=-1, & 
 \mu=4,~ K=\frac{-16}{25},  \\
 \textup{target } {\mathbb E}^3. & \textup{target } \SSS^3. 
\end{array}
$
\caption{Pseudospherical surface of revolution in ${\mathbb E}^3$, and a constant
negative curvature analogue in $\SSS^3$. See also Figure \ref{figure1}. } 
\label{figureintro2}
\end{figure}

\subsection{Examples} Figure \ref{figureintro2} 
shows the well-known pseudospherical surface of revolution, together with a corresponding constant negative curvature surface in $\SSS^3$ obtained, by a different projection, from the same 
loop group frame. 
The surface in $\SSS^3$ is mapped diffeomorphically to $\real^3$, by the stereographic projection, for rendering.
 See Example \ref{example1} below.

Figure \ref{figureintro1} 
shows Amsler's pseudospherical surface in ${\mathbb E}^3$,
which contains   two intersecting straight lines, together with 
a corresponding surface of constant curvature $K=16/25$ in $\SSS^3$, also obtained  from the same loop group frame.    The two straight lines correspond to two great circles.  The great circles appear as straight lines in the image obtained by stereographic projection to $\real^3$.  This example shows that, although the singular sets in the coordinate domain are the same for every surface in the family, the type of singularity 
can change.  The surface obtained at $\mu=-4$ apparently has a swallowtail singularity at a point where the surfaces obtained at $\mu=1$ and $\mu=4$ (See Example \ref{example2} below) each have a cuspidal edge.  This suggests that the singularities of constant curvature surfaces
in $\SSS^3$ are also worth investigating.

\begin{figure}[ht]
\centering
$
\begin{array}{ccc}
\includegraphics[width=55mm]{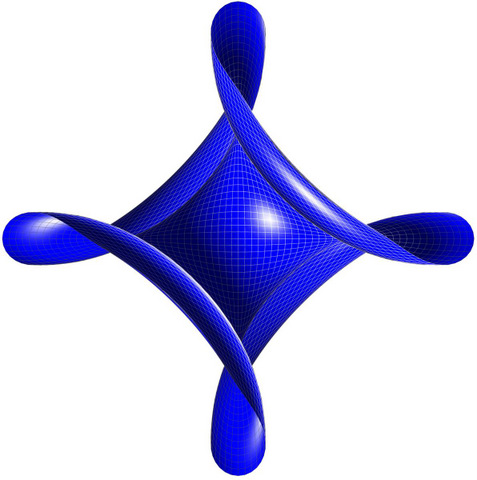}  \quad & \quad
\includegraphics[width=55mm]{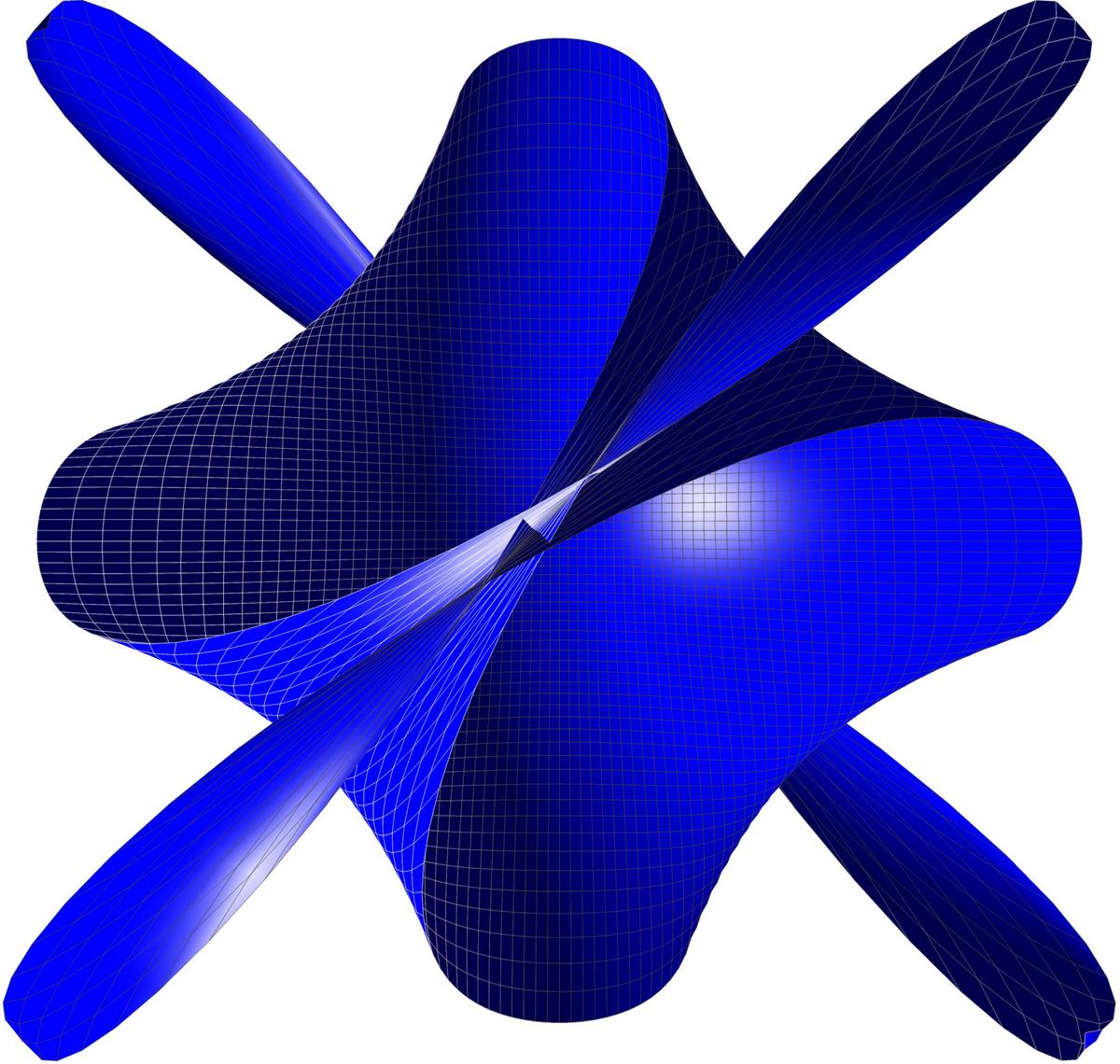}   
 \vspace{2ex} \\
\mu=1, ~ K=-1, & 
 \mu=-4,~ K=\frac{16}{25},  \\
 \textup{target } {\mathbb E}^3. & \textup{target } \SSS^3. 
\end{array}
$
\caption{Amsler's  surface in ${\mathbb E}^3$, and a constant positive curvature analogue in $\SSS^3$. See also Figure \ref{figureA}. } 
\label{figureintro1}
\end{figure}
 \subsection{Comparison with other methods}
It should be noted that Ferus and Pedit \cite{FP2} gave a very nice loop group representation
for isometric immersions of space forms $M_c^n \to \widetilde M_{\tilde c}^{n+k}$ with flat normal bundle
for any $c \neq \tilde c$, with $c \neq 0 \neq \tilde c$.
 Finite type solutions can be generated using the
modified AKS theory described in \cite{FP2}, and all solutions can, in principle,
be constructed from curved flats using the generalized DPW method described in \cite{BD}. For the case
of surfaces, as in the present article, the construction of 
Ferus and Pedit is quite different from the Lorentzian harmonic map approach used here. 
For surfaces, the Lorentzian harmonic map representation is probably more useful, since one obtains, via the
generalized d'Alembert  method, all solutions from essentially arbitrary 
pairs of functions of one variable only: this is the key, for example, to the solution
of the geometric Cauchy problem in \cite{BS1}.
If one were to use the setup in \cite{FP2}, and the generalized DPW method of \cite{BD},
which is the analogue of generalized d'Alembert  method, one instead obtains a curved flat
in the Grassmannian $SO(4)/SO(2) \times SO(2)$ as the basic data, which is not as simple.
In contrast, our basic data are essentially arbitrary functions of one variable.

Another interesting difference between the two approaches is the following: 
we will show below that the loop group frame corresponding to a surface of constant curvature $K<0$ in
$\SSS^3$ also corresponds to a surface with $0<K<1$ in $\SSS^3$, giving some kind of 
Lawson correspondence between two surfaces, one of which has negative curvature
and the other positive.
This correspondence is obtained by evaluating at a 
different value of the loop parameter $\lambda$.   On the other hand,  in 
\cite{BranderAGAG2007}, the loop group maps of Ferus and Pedit are also found
to produce Lawson-type correspondences between various isometric immersions of 
space forms by evaluating in different ranges of $\lambda$.  In this case however,
one does not obtain such a correspondence between surfaces with positive and 
negative curvature.

Finally, we should observe that 
Xia \cite{xia} has also studied isometric 
immersions of constant curvature surfaces in space forms via loop group methods. 
 In that work, for surfaces
in $\SSS^3$, the group $SO(4)$ is used (as opposed to  $SU(2) \times SU(2)$, used here) and 
a loop group representation for the surfaces is given.  However, the generalized d'Alembert  method to construct solutions is not given, and neither is the equivalence of this
problem with Lorentz harmonic maps via the normal Gauss map.   It turns out to be difficult to find
a suitable loop group decomposition in the $SO(4)$ setup used in \cite{xia}, which is 
really the setup for Lorentz harmonic maps into the Grassmannian $SO(4)/SO(2) \times SO(2)$.
The essential problem is that the surfaces in question are not associated to arbitrary 
harmonic maps in the Grassmannian, but very special ones.  In contrast, our use of the group
$SU(2) \times SU(2)$ leads naturally to the normal Gauss map, the harmonicity of which 
is a basic characterization of these surfaces: this leads to a straightforward solution in
terms of the known method for Lorentz harmonic maps.


\section{Preliminaries}\label{sc:Sphere}
\subsection{The symmetric space $\SSS^3$}
 Let $\mathbb{E}^{4}$ be the \textit{Euclidean $4$-space} 
 with standard inner product
$$
\langle \Vec{x},\Vec{y}\rangle=x_{0}y_{0}+x_{1}y_{1}+x_{2}y_{2}+x_{3}y_{3}.
$$ 
We denote by ${e}_{0}=(1,0,0,0)$, ${e}_{1}=(0,1,0,0)$,
${e}_{2}=(0,0,1,0)$, 
${e}_{3}=(0,0,0,1)$ the natural basis of $\mathbb{E}^{4}$.
 
The \textit{orthogonal group} $O(4)$ is defined by
$$
{O}(4)=\{A\in{GL}(4,\mathbb{R})
 \
 \vert
 \
 A^{T}A={I}\}.
$$
Here $I$ is the identity matrix.
We denote by $SO(4)$, the identity component 
of $O(4)$ (called the \textit{rotation group}).

Let us denote by $\mathbb{S}^3$, the unit $3$-sphere 
in $\mathbb{E}^4$ centred at the origin. The unit 3-sphere is 
a simply connected Riemannian space form of constant curvature $1$.

The rotation group $SO(4)$ 
acts isometrically and transitively on $\mathbb{S}^3$ and 
the isotropy subgroup at 
 ${e}_{0}$ is ${SO}(3)$. 
 Hence $\mathbb{S}^{3}={SO}(4)/{SO}(3)$. 
 This representation is a 
 Riemannian symmetric space representation of 
$\mathbb{S}^3$ with involution $\mathrm{Ad}_{\mathrm{diag}(-1,1,1,1)}$.

\subsection{The unit tangent sphere bundle}\label{sc:unit tangent sphere bundle}
Let us denote by $\mathrm{U}\mathbb{S}^3$ the 
\textit{unit tangent sphere bundle} of $\mathbb{S}^3$.
Namely, $\mathrm{U}\mathbb{S}^3$ is the manifold of all unit tangent 
vectors of $\mathbb{S}^3$ and identified with the submanifold
$$
\left\{
(\Vec{x},\Vec{v})\
\vert
\
\langle \Vec{x},\Vec{x}\rangle=
\langle \Vec{v},\Vec{v}\rangle=1,\
\langle \Vec{x},\Vec{v}\rangle=0\
\right\}
$$
of $\mathbb{E}^{4}\times \mathbb{E}^{4}$. 
The tangent space
 $T_{(\Vec{x},\Vec{v})}\mathrm{U}\mathbb{S}^{3}$ at a point
 $(\Vec{x},\Vec{v})$ is expressed as
$$
T_{(\Vec{x},\Vec{v})}\mathrm{U}\mathbb{S}^3=
\{(X,V)\in \mathbb{E}^{4}\times 
\mathbb{E}^{4}
\
\vert
\
\langle \Vec{x},X\rangle=
\langle \Vec{v},V\rangle=0,\ 
\langle\Vec{x}, V \rangle + \langle \Vec{v},
X\rangle =0 
\}.
$$
 Define a $1$-form $\omega$ on $\mathrm{U}\mathbb{S}^3$ by
$$
\omega_{(\Vec{x},\Vec{v})}(X,V)=\langle X,\Vec{v}\rangle=-\langle
\Vec{x},V\rangle.
$$
 Then one can see that $\omega$ is a \textit{contact form} on
 $\mathrm{U}\mathbb{S}^{3}$, \textit{i.e.},
 $(\mathrm{d}\omega)^{2}\wedge \omega\not=0$. The distribution 
$$
\mathcal{D}_{(\Vec{x},\Vec{v})}:=
\{(X,V)\in T_{(\Vec{x},\Vec{v})}\mathrm{U}
\mathbb{S}^3
\
\vert
\
\omega_{(\Vec{x},\Vec{v})}(X,V)=0
\}
$$ 
is called the 
 \textit{canonical contact structure} of
 $\mathrm{U}\mathbb{S}^{3}$.

 The rotation group ${SO}(4)$ 
 acts on $\mathrm{U}\mathbb{S}^3$ via the action
 $A \cdot (\Vec{x}, \Vec{v}) = (A \Vec{x},A\Vec{v})$. It is easy to see that
 under this action the unit tangent sphere bundle
 $\mathrm{U}\mathbb{S}^{3}$ is a homogeneous space 
 of ${SO}(4)$. The isotropy subgroup at
 $(\Vec{e}_{0},\Vec{e}_{1})$ is
$$
\left\{
\left(
\begin{array}{cc}
{I}_2 & 0\\
0 & b 
\end{array}
\right)
\
\biggr
\vert
\
b \in {SO}(2)
\right\}.
$$
Here $I_2$ is the identity matrix of degree $2$.
 Hence $\mathrm{U}\mathbb{S}^3=
 {SO}(4)/{SO}(2)$. The invariant Riemannian metric induced on 
 $\mathrm{U}\mathbb{S}^3={SO}(4)/{SO}(2)$ is a normal homogeneous metric
 (and hence naturally
 reductive), but not Riemannian symmetric. Note that 
 $\mathrm{U}\mathbb{S}^3$ 
 coincides with the \textit{Stiefel manifold} of oriented $2$-frames in $\mathbb{E}^{4}$.
 
\subsection{The space of geodesics}\label{sc:Geo}
 Next we consider $\mathrm{Geo}(\mathbb{S}^3)$ 
 the space of all oriented
 geodesics in $\mathbb{S}^3$.
Take a geodesic $\gamma\in
 \mathrm{Geo}(\mathbb{S}^3)$, then  $\gamma$ is given by the
 intersection of $\mathbb{S}^3$ with an 
oriented $2$-dimensional linear subspace $W$ in $\mathbb{E}^4$.
By identifying $\gamma$ with $W$,
 the space $\mathrm{Geo}(\mathbb{S}^3)$ is identified with the 
 Grassmann manifold
 $\mathrm{Gr}_{2}(\mathbb{E}^{4})$ of oriented $2$-planes in 
 Euclidean $4$-space. The natural projection 
 $\pi_{1}:\mathrm{U}\mathbb{S}^3\to \mathrm{Geo}(\mathbb{S}^3)$ is 
 regarded as the map 
 $$
 \pi_{1}(\Vec{x},\Vec{v})=\mbox{the geodesic}\>\gamma \>\mbox{satisfying the condition}\>
 \gamma(0)=\Vec{x},\ \gamma^{\prime}(0)=\Vec{v}.
 $$
The rotation group $SO(4)$ acts 
isometrically and transitively on $\mathrm{Geo}(\mathbb{S}^3)$. 
The isotropy subgroup 
 at ${e}_{0}\wedge{e}_{1}$ is ${SO}(2)\times {SO}(2)$.

 Therefore, the tangent space $T_{{e}_{0}\wedge{e}_{1}}
 \mathrm{Geo}(\mathbb{S}^3)$
 is identified with the linear subspace 
$$
\left\{
\left(
\begin{array}{cccc}
0 & 0 & -x_2 & -x_3\\
0 & 0 & -x_{21} & -x_{31}\\
x_{2} & x_{21} & 0 & 0\\
x_{3} & x_{31} & 0 & 0
\end{array}
\right)\right\}
$$
of $\mathfrak{so}(4)$.
The standard invariant complex structure $J$  
on $\mathrm{Geo}(\mathbb{S}^3)=\mathrm{Gr}_{2}(\mathbb{E}^{4})$ is 
given explicitly by
$$
J
\left(
\begin{array}{cccc}
0 & 0 & -x_2 & -x_3\\
0 & 0 & -x_{21} & -x_{31}\\
x_{2} & x_{21} & 0 & 0\\
x_{3} & x_{31} & 0 & 0
\end{array}
\right)
=
\left(
\begin{array}{cccc}
0 & 0 & x_{21} & x_{31}\\
0 & 0 & -x_{2} & -x_{3}\\
-x_{21} & x_{2} & 0 & 0\\
-x_{31} & x_{3} & 0 & 0
\end{array}
\right).
$$
One can see that $\mathrm{Gr}_{2}(\mathbb{E}^{4})$
is a Hermitian symmetric space with Ricci tensor $2\langle\cdot,
\cdot\rangle$. The K{\"a}hler form $\Omega$ is related to the contact form 
$\omega$ by $\pi_{1}^{*}\Omega
=d\omega$.

\section{Surface theory in $\mathbb{S}^3$}\label{sc:SurfacesS3}

\subsection{The Lagrangian and Legendrian Gauss maps}
 Let $f:M\to \mathbb{S}^3\subset \mathbb{E}^{4}$ be a conformal
 immersion of a Riemann surface with unit normal vector field $n$. 
 Then we define the (\textit{Lagrangian}) \textit{Gauss map} $L$ of $f$ by 
$$
L:=f \wedge n :M \to \mathrm{Gr}_{2}(\mathbb{E}^4).
$$ 
One can see that $L$ is an immersion and in addition,
it is \textit{Lagrangian} with respect to the canonical 
symplectic form $\Omega$ of $\mathrm{Gr}_{2}(\mathbb{E}^4)$, 
\textit{i.e.}, $L^{*}\Omega=0$. Under the identification  $\mathrm{Gr}_{2}(\mathbb{E}^4)
=\mathrm{Geo}(\mathbb{S}^3)$, the Lagrangian Gauss map is referred as the
oriented normal geodesic of $f$ (and called the \textit{spherical Gauss map}). 
  
On the other hand, we have a map $\mathcal{L}:=(f,n):M\to \mathrm{U}\mathbb{S}^3$. 
This map is \textit{Legendrian} with respect to the canonical contact structure of 
$\mathrm{U}\mathbb{S}^3$,
that is, $\mathcal{L}^{*}\omega=0$. This map $\mathcal{L}$ is called 
the \textit{Legendrian Gauss map} of $f$.
 
\subsection{Parallel surfaces} 
An oriented \textit{geodesic congruence} in $\mathbb{S}^3$ is an immersion 
of a $2$-manifold $M$ into the space $\mathrm{Geo}(\mathbb{S}^3)$ of geodesics. 
Now let $f:M\to \mathbb{S}^3$ be a surface with unit normal $n$.
Then a \textit{normal geodesic congruence} through $f$ at a distance $r$ is 
the map $f^r:M\to \mathbb{S}^3$ defined by
$$
f^r:=\cos r {f}+\sin {r}n.
$$ 
If $f$ satisfies the condition
$
\cos (2r)-\sin (2r)H+\sin^{2}(r)K\not=0
$
then $f^r$ is an immersion. Here $H$ and $K$ are the 
mean and Gauss curvatures of $f$, 
respectively. 
If  $f^r$ is an immersion, then it is called the 
\textit{parallel surface} of $f$ at the distance $r$.
The correspondence $f\longmapsto f^r$ is called the
\textit{parallel transformation}.

 \subsection{Legendrian lifts, frontals and fronts.}
 The Gauss map $L$ of an oriented surface 
 $f:M
\to \mathbb{S}^3$ with unit normal $n$ is a Lagrangian immersion
into $\mathrm{Gr}_{2}(\mathbb{E}^4)$. 
Conversely, we have the following fact (see \cite{Palmer}): 
\begin{proposition} Let $L:M\to \mathrm{Gr}_{2}(\mathbb{E}^4)$ be 
a Lagrangian immersion. Then, locally, $L$ is a projection of 
a Legendrian immersion $\mathcal{L}:M \to \mathrm{U}\mathbb{S}^3$.
The Legendrian immersion is unique up to parallel transformations.
\end{proposition} 
The Legendrian immersion $\mathcal{L}$ is called  
a \textit{Lie surface} in \cite{Palmer}. 
If $f:M\to \mathbb{S}^3$
is an immersion with unit normal $n$, then 
$\mathcal{L}:=(f,n)$ is a Legendrian immersion into $\mathrm{U}\mathbb{S}^3$.
However, even if $\mathcal{L}$ is a Legendrian immersion, then
$f:=\pi_{2}\circ \mathcal{L}$ need not be an immersion 
although it possesses a unit normal $n$. 
Here $\pi_{2}:\mathrm{U}\mathbb{S}^3\to \mathbb{S}^3$ is the 
natural projection.
\begin{remark}
{\rm
A smooth map $f:M\to \mathbb{S}^3$ is 
called a \textit{frontal} if for any point 
$p\in M$, there exists a neighborhood $\mathcal{U}$ of $p$ and 
a unit vector field $n$ along $f$ defined on $\mathcal{U}$ 
such that $\langle \mathrm{d}f,n\rangle=0$. 
A frontal is said to be 
\textit{co-orientable} if there exists a unit vector field $n$ along 
$f$ such that $\langle \mathrm{d}f,n\rangle=0$. 
Namely a co-orientable frontal is a smooth map 
$f:M\to \mathbb{S}^3$ which has a lift $\mathcal{L}=(f,n)$ 
to $\mathrm{U}\mathbb{S}^3$ satisfying the 
\textit{Legendrian condition} $\mathcal{L}^{*}\omega=\langle\mathrm{d}f,n\rangle=0$.
A co-orientable frontal is called a \textit{front} if its 
Legendrian lift is an immersion.

Our main interest is surfaces of constant curvature 
$K<1$ in $\mathbb{S}^3$.
Except the case $K=0$, any surface of constant 
Gauss curvature $K<{1}$ has singularities.
A theory of the singularities of fronts can be found in 
Arnold \cite{Arnold}. 
Geometric concepts, such as curvature and 
completeness, for surfaces with singularities 
have been defined by Saji, Umehara and Yamada in \cite{SUY}.
}
\end{remark}
 \subsection{Asymptotic coordinates}
 Hereafter assume that the Gaussian curvature $K$ is less than $1$.
 This implies that the second fundamental form $\mathrm{II}$ derived from $n$ is
 a possibly singular Lorentzian metric on $M$.

 Represent $K$ as $K=1-\rho^2$ for some positive 
 function $\rho$ and take a local \textit{asymptotic 
coordinate system} $(u,v)$ defined on a simply 
connected domain $\mathbb{D}\subset M$.
Then the first fundamental form $\mathrm{I}$ and second 
fundamental form $\mathrm{II}$ are given by (see \textit{eg.} \cite{Moore})
\beq \label{firstandsecondff}
\mathrm{I}=A^2\> \mathrm{d}u^2+2AB\cos \phi\>\mathrm{d}u\mathrm{d}v+B^2\> \mathrm{d}v^2,
\ \
\mathrm{II}=2{\rho}AB\sin \phi\> \mathrm{d}u\mathrm{d}v.
\eeq
Note that asymptotic coordinates $(u,v)$ are \textit{conformal} with 
respect to the second fundamental form. 
We may regard $M$ as a (singular) \textit{Lorentz surface} (\cite{Klotz3}) 
with respect to the conformal structure determined by 
$\mathrm{II}$ (called the \textit{second conformal structure} \cite{Klotz}, \cite{Klotz2}). 
Thus one can see that 
$$
C=A^2\>\mathrm{d}u^2+B^2\>\mathrm{d}v^2
$$
is well defined on $M$. 

The \textit{Gauss equation} is given by
$$
\phi_{uv}-
\left(
\frac{\rho_v}{2\rho}\frac{B}{A}\sin \phi\right)_{u}
-
\left(
\frac{\rho_u}{2\rho}\frac{A}{B}\sin \phi\right)_{v}
+(1-\rho^2)AB\sin \phi=0.
$$
Now we introduce functions $a$ and $b$ by 
$a=A\rho$ and $b=B\rho$. The 
\textit{Codazzi equations} are
\beq \label{codazzieqns}
a_{v}-\frac{\rho_v}{2\rho}a+\frac{\rho_u}{2\rho}b\cos \phi=0,
\ \
b_{u}-\frac{\rho_u}{2\rho}b+\frac{\rho_v}{2\rho}a\cos \phi=0.
\eeq
The Codazzi equations imply that if $K$ is constant, then 
we have $a_v=b_u=0$. In addition, 
the Gauss-Codazzi equations are 
invariant under the deformation:
$$
a\longmapsto \lambda{a}, \ \
b\longmapsto \lambda^{-1}b,  \ 
\lambda\in \mathbb{R}^{*}:=\mathbb{R}\setminus \{0\}.
$$
Thus there exists a one-parameter deformation $\{f_\lambda\}_{\lambda\in \mathbb{R}^*}$ 
of $f$ preserving the second fundamental form and the Gauss curvature. 
The resulting family is called the \textit{associated family} of $f$. 
The existence of the associated family motivates us to study constant Gauss curvature 
surfaces in $\mathbb{S}^3$ by loop group methods.

\section{The loop group formulation}
\subsection{The $SU(2)\times SU(2)$ frame}
Let us now identify $\SSS^3$ with $SU(2)$, via 
\bdm
 (z,w) \in \SSS^3 \subset \real^4=\C^2 \quad  \longleftrightarrow  \quad \bbar{cc} z & w \\ -\bar w & \bar z\ebar \in SU(2).
\edm
The standard metric $g$ on $\SSS^3$ is then given by left translating  $V, W \in T_x \SSS^3$ to the tangent space at the identity, $T_e SU(2) = \mathfrak{su}(2)$, i.e.
\bdm
g(V,W) := \left < x^{-1} V , x^{-1} W \right >,
\edm
where the inner product on $\mathfrak{su}(2)$ is given by $\left< X , Y \right> = -\textup{Tr}(XY)/2$.
The natural basis $\{e_0,e_1,e_2,e_3\}$ of $\mathbb{E}^4$ is identified with   
\bdm
e_0=e=I_2, \ 
e_1 = \bbar{cc} 0 & 1 \\ -1 & 0 \ebar, \quad e_2 = \bbar{cc} 0 & i \\ i & 0 \ebar, \quad
 e_3 = \bbar{cc} i & 0 \\ 0 & -i\ebar.
\edm
Note that $\{e_1,e_2,e_3\}$ is an orthonormal basis of $\mathfrak{su}(2)$.
We have the commutators $[e_1, e_2] = 2 e_3$, $[e_2, e_3] = 2 e_1$ and $[e_3, e_1] = 2 e_2$, so that
the cross product on ${\mathbb E}^3$ is given by $A \times B = \frac{1}{2}[A,B]$. 
Note that $\mathbb{S}^3$ is represented by $(SU(2)\times SU(2))/SU(2)$ as a 
Riemannian symmetric space. The natural projection is given by 
$(G,F)\longmapsto 
GF^{-1}$.

Let $M$ be a simply connected $2$-manifold, and suppose given an immersion $f: M \to \SSS^3$, with global asymptotic coordinates $(u,v)$,
and first and second fundamental fundamental forms as above at (\ref{firstandsecondff}). Set $\theta=\phi/2$ and
\beqas
\xi_1 = \cos (\theta) e_1 - \sin (\theta) e_2 = \bbar {cc} 0 & e^{-i\theta} \\ -e^{ i\theta} & 0 \ebar,\\
\xi_2 = \cos (\theta) e_1 + \sin (\theta) e_2 = \bbar {cc} 0 & e^{i\theta} \\ -e^{ -i\theta} & 0 \ebar.
\eeqas
Then $\left <\xi_1, \xi_2\right> = \cos \phi$, and so we can define a map $F: M \to SU(2)$ by the equations
\beq  \label{coordframedef}
f^{-1} f_u = A \, \Ad _F \xi_1, \quad f^{-1} f_v = B \, \Ad _F \xi_2, \quad f^{-1} n = \Ad _F e_3,
\eeq
where $n$ is the unit normal given by $n= (AB)^{-1}f (f^{-1} f_u \times  f^{-1} f_v)$.

Setting $G = f F$, the map $\mathcal{F} = (F, G) : M \to SU(2) \times SU(2)$ is a lift of $f$, and the projection to $SU(2)$ 
is given by
\bdm
f = GF^{-1}.
\edm
We call $\mathcal{F}$ the \emph{coordinate frame} for $f$. We now want to get expressions for the Maurer-Cartan forms of $F$ and $G$.
Differentiating $G = fF$, and substituting in the expressions at (\ref{coordframedef}) for $f^{-1} f_u$ 
and $f^{-1} f_v$, we obtain
\beq
\begin{split}
\label{FGdifference}
G^{-1} G_u - F^{-1} F_u = A \xi_1, \\
 G^{-1} G_v - F^{-1} F_v = B\xi_2.
 \end{split}
\eeq
Now write $F^{-1} F_u = a_1 e_1 + a_2 e_2 + a_3 e_3$.
  Differentiating the expression $f^{-1} f_u = A \Ad_F \xi_1$, we obtain
\beqas
f^{-1} f_{uu} &=& A^2 \Ad_F \xi_1^2 + \frac{\partial A}{\partial u} \Ad_F \xi_1 + A \, \Ad_F [F^{-1} F_u , \xi_1] + A \, \Ad_F \frac{\partial\xi_1}{\partial u} \\
&=& A \, \Ad_F \left( -A e_0 + A^{-1}  \frac{\partial A}{\partial u} \xi_1  \right.\\
   &&  \quad \quad   + \left. [a_1 e_1 + a_2 e_2 + a_3 e_3 , \cos \theta e_1 - \sin \theta e_2] +  \frac{\partial \xi_1}{\partial u}\right) \\
&=& A(-2 a_1 \sin \theta -2 a_2 \cos \theta) \, \Ad_F e_3 + \Ad_F ( d_0 e_0 + d_1 e_1 + d_2 e_2),
\eeqas
where we are only interested in the coefficient of $\Ad_F e_3$, that is, of  $f^{-1} n$. Since the second fundamental form is assumed to be $\mathrm{II} = 2 \rho AB \sin \phi \dd u \dd v$, we know that $\left< f^{-1}n , f^{-1}f_{uu} \right> = 0$.  Hence the coefficient of $n$ in the above equation is zero: $0 = A(-2 a_1 \sin \theta -2 a_2 \cos \theta)$, or
\bdm
a_2 = -a_1 \tan \theta.
\edm
Next, differentiating $f^{-1} f_v = B \, \Ad_F \xi_2$ with respect to $u$, we deduce:
\bdm
f^{-1} f_{uv} = AB \, \Ad_F (\xi_1 \xi_2) + \Ad_F \left(\frac{\partial B}{\partial u} \xi_2 + B [F^{-1}F_u, \xi_2] + B \frac{\partial \xi_2}{\partial u}\right), 
\edm
and the coefficient of $\Ad_F e_3$ on the right hand side is $AB \sin 2\theta + B(2a_1 \sin \theta - 2 a_2 \cos \theta)$.  Substituting in $a_2 = -a_1 \tan \theta$, the equation
 $\left< f^{-1}n , f^{-1}f_{uv} \right> =  \rho AB \sin \phi$ then becomes
 \bdm
 \rho AB \sin 2 \theta = AB \sin 2 \theta + B 4 a_1 \sin \theta.
 \edm
Hence, $a_1 = A(\rho-1) \cos(\theta)/2$, and $a_2 = -A(\rho-1) \sin(\theta)/2$. Writing $U_0 := a_3 e_3$,
we have:
\bdm
F^{-1} F_u = U_0 + \frac{\rho-1}{2}A \xi_1.
\edm
From the equations (\ref{FGdifference}) we also have:
\bdm
G^{-1} G_u = U_0 + \frac{\rho+1}{2} A \xi_1.
\edm
Similarly, one obtains the expressions: 
$ F^{-1} F_v = V_0 -\frac{\rho+1}{2} B \xi_2$ and
   $G^{-1} G_v = V_0 - \frac{\rho-1}{2} B \xi_2$, where $V_0$ is a scalar times $e_3$.

The Maurer-Cartan form, $\alpha = F^{-1} \dd F$, of $F$ thus has the expression
\begin{flalign} \label{alphaexpression}
\begin{split}
  & \alpha =  \alpha_0 + \alpha_1 + \alpha_{-1}, \\
& \alpha_0 = U_0 \dd u + V_0 \dd v, \quad \alpha_1 = \frac{\rho-1}{2} A \xi_1 \dd u,
\quad  \alpha_{-1} = -\frac{\rho+1}{2} B \xi_2 \dd v,
\end{split}
\end{flalign}
 and similarly,   $G^{-1} \dd G = \beta = \beta_0 + \beta_1 + \beta_{-1}$, with
\beq  \label{betaexpression}
\beta_0 = U_0 \dd u + V_0 \dd v, \quad \beta_1 = \frac{\rho+1}{2} A \xi_1 \dd u,
\quad  \beta_{-1} = - \frac{\rho-1}{2} B \xi_2 \dd v.
\eeq
One can check that $U_0$ and $V_0$ are of the form 
$U_{0}=-\theta_{u}e_3/2$ and $V_{0}=\theta_{v}e_{3}/2$.

\subsection{Ruh-Vilms property}
Now we investigate Lorentz harmonicity, with respect to the 
second conformal structure,
of the \textit{normal Gauss map} $\nu=f^{-1}n$ of $f$. 
By definition, $\nu$ takes value in the unit $2$-sphere 
$\mathbb{S}^{2}=\mathrm{Ad}_{SU(2)}e_3$ in the Lie algebra $\mathfrak{su}(2)$.
Since $f$ and $n$ are given by $f=GF^{-1}$, $\nu=\mathrm{Ad}_{F}e_3$,
 we have 
$$
\nu_{u}=\mathrm{Ad}_{F}[U,e_3], \ \
\nu_{v}=\mathrm{Ad}_{F}[V,e_3],
$$
where $U=F^{-1}F_u$ and $V=F^{-1}F_v$.
From these we have
\begin{align*}
\frac{\partial}{\partial{v}}\nu_{u}&=\mathrm{Ad}_{F}
\left(
\left[\frac{\partial{U}}{\partial{v}},e_3\right]
+\left[V,[U,e_3]\>\right]
\right),
\\
\frac{\partial}{\partial{u}}\nu_{v}&=\mathrm{Ad}_{F}
\left(
\left[\frac{\partial{V}}{\partial{u}},e_3\right]
+\left[U,[V,e_3]\>\right]
\right).
\end{align*}
Next we have
$$
\left[\frac{\partial{U}}{\partial{v}},e_3\right]
=\{A(\rho-1)\sin\theta\}_{v}e_{1}+
\{A(\rho-1)\cos\theta\}_{v}e_{2},
$$
$$
\left[\frac{\partial{V}}{\partial{v}},e_3\right]
=-\{B(\rho+1)\sin\theta\}_{u}e_{1}+
\{B(\rho+1)\cos\theta\}_{u}e_{2},
$$
$$
[V,[U,e_3]\>]=-A(\rho-1)\theta_{v}(\cos \theta{e}_1-\sin\theta{e}_2)
+\frac{1}{2}AB(\rho^2-1)\cos(2\theta)e_3,
$$
$$
[U,[V,e_3]\>]=B(\rho+1)\theta_{u}(\cos \theta{e}_1+\sin\theta{e}_2)
+\frac{1}{2}AB(\rho^2-1)\cos(2\theta)e_3.
$$
Here we recall that a smooth map $\nu:M\to \mathbb{S}^2\subset \mathbb{E}^3$
of a Lorentz surface $M$ into the $2$-sphere is 
said to be a \textit{Lorentz harmonic map} (or \textit{wave map}) if 
its tension field with respect to any Lorentzian metric in the
 conformal class vanishes. This is equivalent to the existence of
  a function $k$ such that 
$$
\nu_{uv}=k\nu
$$ 
for any conformal coordinates $(u,v)$.

First $(\nu_u)_v$ is parallel to $\nu$ if and only if
$$
A_{v}(\rho-1)+A\rho_v=0.
$$
Inserting the Codazzi equation (\ref{codazzieqns}) into this, we get 
\begin{equation}\label{RV-1}
a(\rho+1)\rho_{v}- b (\rho-1) \cos\phi\rho_u=0.
\end{equation}
Analogously, $(\nu_v)_u$ is parallel to $\nu$ if and only 
if 
$$
B_u(1+\rho)+B\rho_u=0.
$$
Inserting the Codazzi equation 
again, we get
\begin{equation}\label{RV-2}
b(\rho-1)\rho_{v}-a(\rho+1)(\cos \phi) \rho_v=0.
\end{equation}
Thus $\nu$ is Lorentz harmonic if and only if (\ref{RV-1}) and 
(\ref{RV-2}) hold.
The system (\ref{RV-1})-(\ref{RV-2}) can be written in 
the following matrix form:
$$
\left(
\begin{array}{cc}
b(\rho-1) & -a (\rho + 1)\cos \phi\\
-b(\rho-1)\cos \phi & a(\rho + 1)
\end{array}
\right)
\left(
\begin{array}{c}
\rho_u
\\
\rho_v
\end{array}
\right)
=\left(
\begin{array}{c}
0
\\
0
\end{array}
\right).
$$
The determinant of the coefficient matrix is 
computed as $ab(\rho^2-1)\sin^{2}\phi$. Thus under the condition 
$\rho\not=1$, \textit{i.e.}, $K\not=0$, we have 
$\nu$ is Lorentz harmonic if and only if $K$ is constant.

In case $\rho=1$, then we have $U=-\theta_{u}e_3/2$ and so 
$\nu_u=\mathrm{Ad}_{F}[U,e_3]=0$. Hence $\nu_{uv}=0$. 
Thus $g$ is Lorentz harmonic.
\begin{theorem}
Let $f:M\to \mathbb{S}^3$ be an isometric immersion of Gauss curvature $K<1$. 
Then the normal Gauss map $\nu$ is Lorentz harmonic with respect to the 
conformal structure determined by the second fundamental form if and only if 
$K$ is constant.
\end{theorem}
This characterization is referred as the \textit{Ruh-Vilms property} 
for constant curvature surfaces in $\mathbb{S}^3$ with $K<1$. 

\begin{remark}{\rm
Under the identification 
$\mathbb{S}^3=(SU(2)\times SU(2))/SU(2)$, 
the space $\mathrm{Geo}(\mathbb{S}^3)$ is identified 
with the Riemannian product $\mathbb{S}^2\times \mathbb{S}^2=(SU(2)\times SU(2))/(U(1)\times U(1))$.
Moreover the Lagrangian Gauss map $L$ corresponds to the map (\textit{cf.} \cite{AA}, \cite{Kitagawa2}):
$$
L\longleftrightarrow (nf^{-1},f^{-1}n)=(\mathrm{Ad}_{G}e_3,\mathrm{Ad}_{F}e_3).
$$
Thus the Ruh-Vilms property can be rephrased as follows.
}
\end{remark} 
\begin{corollary}
Let $f:M\to \mathbb{S}^3$ be an isometric immersion of Gauss curvature $K<1$. 
Then the Lagrangian Gauss map $L$ is Lorentz harmonic with respect to the 
conformal structure determined by the second fundamental form if and only if 
$K$ is constant.
\end{corollary}
The Legendrian Gauss map has the formula $\mathcal{L}=(f,n)=(GF^{-1},Ge_{3}F^{-1})$.
\begin{remark}
{\rm For an oriented minimal surface $f:M\to \mathbb{S}^3$ with unit normal $n$.
Then its Lagrangian Gauss map $L=f\wedge n$ is a \textit{harmonic map} with respect to the 
conformal structure determined by the \textit{first fundamental form}. Hence 
$L$ is a minimal Lagrangian surface in 
the Grassmannian \cite[Proposition 3.1]{Palmer}, see also \cite{CU}.
}
\end{remark}
\subsection{The loop group formulation for constant curvature surfaces}
Let $\alpha$ and $\beta$ be as defined above at (\ref{alphaexpression}) and (\ref{betaexpression}).
Let us now define the family of $1$-forms 
\bdm
\alpha^\lambda = \alpha_0 + \lambda \alpha_1 + \lambda^{-1} \alpha_{-1},
\edm
where $\lambda$ is a complex parameter.
The integrability conditions for the $1$-forms $\alpha$ and $\beta$ are 
$\dd \alpha + \alpha \wedge \alpha=0$
and $\dd \beta + \beta \wedge \beta=0$. 
Using these two equations, which must already be satisfied, 
it is fairly straightforward to deduce that $\alpha^\lambda$ is 
integrable for all $\lambda$ if and only if $\rho$ is \textit{constant}, 
in other words, if and only if the immersion $f$ has \emph{constant} curvature $1-\rho^2$.
In this case we have, of course, $\alpha = \alpha^1$, but we also have
\beq  \label{betaeqn}
\beta = \alpha^\mu, \quad \textup{where} \quad \mu = \frac{\rho+1}{\rho-1}.
\eeq
From now on, we assume that $\rho$ is constant, so that $\dd \alpha^\lambda + \alpha^\lambda \wedge \alpha^\lambda = 0$ for all non-zero complex values of $\lambda$.
Let us choose coordinates for the ambient space such that
\beq \label{initialcond}
F(u_0,v_0) = f(u_0,v_0)=G(u_0,v_0)= I,
\eeq
 at some base point $(u_0,v_0)$.
We further assume that $M$ is simply connected.  Then we can integrate the equations
\bdm
\hat F^{-1} \dd \hat F = \alpha^\lambda, \quad \hat F(u_0,v_0) = I,
\edm
 to obtain a 
map $\hat F: M \to \mathcal{G} = \Lambda SL(2,\C)_{\sigma \tau}$. 
Here the \textit{twisted loop group} $\Lambda SL(2,\C)_{\sigma\tau}$ is the fixed point subgroup of
the free loop group $\Lambda SL(2,\C)$, by the involutions $\sigma$ and $\tau$, 
that are defined as follows:
$$
\sigma x(\lambda): = \textup{Ad}_{\diag(1,-1)}(x(-\lambda)),
\ \ 
\tau x (\lambda) := \overline{x((\bar \lambda))^t}^{-1}. 
$$ 
Elements of
$\mathcal{G}$ take values in $SU(2)$ for real values of $\lambda$.

By definition we have $F = \hat F  |_{\lambda=1}$, and moreover, from (\ref{betaeqn}) and
the initial condition (\ref{initialcond}), we also have $G = \hat F |_{\lambda=\mu}$.  
Thus $\hat F$ can be thought of as a lift of the coordinate frame $\mathcal{F} = (F,G)$, 
with the 
projections $(F,G) = \left( \hat F |_{\lambda=1} , \hat F |_{\lambda=\mu} \right)$ and
\beq  \label{projectionformula}
f = \hat F \big|_{\lambda=\mu} \, \hat F^{-1} \big|_{\lambda=1}.
\eeq
Thus we may call the map $\hat F$ the \emph{extended coordinate frame} for $f$.

 Let us now consider a general map into the twisted loop group
 $\mathcal{G}$ that has a similar Maurer-Cartan form to $\alpha^\lambda$:  
 first let $K$ be the diagonal subgroup of $SU(2)$ and 
  $\mathfrak{su}(2)= \mathfrak{k} + \mathfrak{p}$ be the symmetric space decomposition,
  induced by $\mathbb{S}^2=G/K=SU(2)/U(1)$,  of
  the Lie algebra, that is 
  \bdm
  \mathfrak{k} = \textup{span}(e_3), \quad \textup{and} \quad
  \mathfrak{p} = \textup{span}(e_1, e_2).
  \edm
  
 \begin{definition}  \label{admissibleframedef}
 Let $M$ be a simply connected subset of $\real^2$ with coordinates $(u,v)$.
 An \emph{admissible frame} is a smooth map $\hat F: M \to \mathcal{G}$, the Maurer-Cartan form of
 which has the Fourier expansion:
 \bdm
 \hat F^{-1} \dd \hat F = \alpha_0 + \lambda B_1 \, \dd u + \lambda^{-1} B_{-1} \, \dd v,
 \quad \quad \alpha_0 \in  \mathfrak{k} \otimes \Omega^1(M),
   \quad B_{\pm 1}(u,v) \in \mathfrak{p}.
 \edm
 The admissible frame is \emph{regular at a point $p$}, if  $B_1(p)$ and $B_{-1}(p)$ are linearly independent, and $\hat F$ is called \emph{regular} if it is regular at every point.
 \end{definition}
Note that the extended coordinate frame for a constant curvature $1-\rho^2$ immersion, defined above, is a regular admissible frame on $M$.  Conversely, we have the following:

\begin{lemma}  \label{converselemma}
Let $\hat F: M \to \mathcal{G}$ be a regular admissible frame.  
Let $\mu$ be any real number not equal to $1$ or $0$. 
Then the map $f: M \to \SSS^3 = SU(2)$, defined by the projection (\ref{projectionformula}),
is an immersion of constant curvature
\bdm
K_\mu = 1- \rho^2, \quad \textup{where} \quad \rho:= \frac{\mu+1}{\mu-1}.
\edm
The first and second fundamental form are given by
\beq \label{firstandsecondff2}
\begin{split}
\mathrm{I}=A^2\> \mathrm{d}u^2+2AB\cos \phi\>\mathrm{d}u\mathrm{d}v+B^2\> \mathrm{d}v^2,
\quad
\mathrm{II}=2{\rho}AB\sin \phi\> \mathrm{d}u\mathrm{d}v,\\
\textup{where} \quad A = (\mu-1) |B_1|, \quad  \quad B = (\mu^{-1}-1)|B_{-1}|,
\end{split}
\eeq
and $\phi$ is the angle between $B_{1}$ and $B_{-1}$.
\end{lemma}

 \begin{proof}
 Set 
 \bdm 
 F:= \hat F |_{\lambda=1}, \quad  G:=\hat F  |_{\lambda=\mu},
 \edm
  so that $f= GF^{-1}$.
 Differentiating this formula and using the expressions 
 $F^{-1} \dd F = \alpha_0 + B_1 \, \dd u +  B_{-1} \, \dd v$ and
 $G^{-1} \dd G = \alpha_0 + \mu B_1 \, \dd u + \mu^{-1} B_{-1} \, \dd v$, we obtain
 \beq \label{fufvderivatives}
 f^{-1} f_u = (\mu-1) \Ad_F B_1 , \quad \quad f^{-1} f_v = (\mu^{-1}-1)\Ad_F B_{-1}.
 \eeq
 Thus, since $B_{\pm 1}$ are linearly independent for a regular admissible frame, 
 the map $f$ is an immersion and the first fundamental form is given by:
 \bdm
 (\mu-1)^2 |B_1|^2 \dd u^2 + 2 (\mu-1)(\mu^{-1}-1) \cos (\phi) |B_1||B_{-1}| \dd u\, \dd v
  + (\mu^{-1}-1)^2 |B_{-1}|^2 \dd v^2,
\edm
 where $\phi$ is the angle between $B_{1}$ and $B_{-1}$.  This gives the formula at
 (\ref{firstandsecondff2}) for the first fundamental form. 
 
 It remains to show the formula at
 (\ref{firstandsecondff2}) for the \emph{second} fundamental form, from which it will follow that the
 intrinsic curvature is $1-\rho^2$.  Since $B_{\pm 1}$ take values in $\mathfrak{p}$,
 and $e_3$ is perpendicular to $\mathfrak{p}$,
  it follows from the equations at (\ref{fufvderivatives}), that a choice of unit normal is given by 
\bdm
n = f \Ad_F e_3.
\edm
  Differentiating the equations (\ref{fufvderivatives}) then leads to 
 \beqas
 \left< f^{-1} f_{uu} , f^{-1} n \right> &=& \left< f^{-1} f_{vv} , f^{-1} n \right> = 0,\\
 \left< f^{-1} f_{uv} , f^{-1} n \right> &=& (1-\mu)(1+\mu^{-1}) |B_1||B_{-1}|  \sin \phi \\
  &=& \rho \,(\mu-1)(\mu^{-1}-1) |B_1||B_{-1}| \sin \phi,
 \eeqas
 which gives the formula at (\ref{fufvderivatives}) for $\mathrm{II}$.
 
 \end{proof}
 
 Note that
\beqas
\textup{for } \mu<0:   & \quad K_\mu \in (0,1];  \quad & \quad K_{-1}=1;\\
\textup{for } \mu> 0:&  \quad K_\mu <0; \quad & \quad \lim_{\mu \to 1} K_\mu =-\infty.
\eeqas

The Legendrian Gauss map and Lagrangian Gauss map of 
$f=\hat{F}_{\lambda=\mu}\hat{F}_{\lambda=1}^{-1}$
are given by
$$
\mathcal{L}=(\hat{F}_{\lambda=\mu}\hat{F}_{\lambda=1}^{-1}, 
\hat{F}_{\lambda=\mu}\>e_{3}\>\hat{F}_{\lambda=1}^{-1}),
\ \
L=(\mathrm{Ad}_{\hat{F}_{\lambda=\mu}}e_3,\mathrm{Ad}_{\hat{F}_{\lambda=1}}e_3),
$$
respectively.

\subsection{The generalized d'Alembert representation}
As we have shown,
the problem of finding a non-flat constant curvature immersion 
$f:M \to \mathbb{S}^3$ with $K<1$ is 
equivalent to finding an admissible frame. 
As a matter of fact, Definition \ref{admissibleframedef} of an admissible frame is  identical to the extended
$SU(2)$ frame for a pseudospherical surface in the Euclidean space $\mathbb{E}^3$ 
(see, for example, \cite{BS1, DoSter, todaagag}). The surfaces in ${\mathbb E}^3$ are obtained from the same frame, not
by the projection (\ref{projectionformula}), but by the so-called \emph{Sym formula}.
 We will explain the connection between these problems in the next section, but 
the  point we are making here is that the problem of constructing these admissible frames
by the generalized d'Alembert representation has already been solved in \cite{todaagag}.

A presentation of the method, using similar definitions to those found here,
can be found in \cite{BS1}. The basic data used to construct any admissible frame is:
\begin{definition}[\cite{BS1}, Definition 5.1]
Let $I_u$ and $I_v$ be two real intervals, 
with coordinates $u$ and $v$, respectively. 
A potential pair $(\eta_{+},\eta_{-})$ 
is a pair of smooth $\Lambda \mathfrak{sl}(2,\C)_{\sigma\tau}$-valued 
$1$-forms on $I_u$ and $I_v$ respectively with Fourier expansions in $\Lambda$ 
as follows{\rm:}
$$
\eta_{+}=\sum_{j=-\infty}^{1}(\eta_{+})_{j}\lambda^{j}\>\mathrm{d}u,
\ \
\eta_{-}=\sum_{j=-1}^{\infty}(\eta_{-})_{j}\lambda^{j}\>\mathrm{d}v.
$$
The potential pair is called {\em regular} if 
$[(\eta_{+})_1]_{12}\not=0$ and 
$[(\eta_{-})_{-1}]_{12}\not=0$.
\end{definition}
The admissible frame $\hat F$ is then obtained by solving $F_\pm ^{-1} \dd F_\pm = \eta_\pm$, with 
initial conditions $F_\pm(0) = I$, thereafter performing, at each $(u,v)$, a 
\textit{Birkhoff decomposition} \cite{LoopGroup}  
$$
F_+^{-1}(u) F_-(v) = H_-(u,v) H_+(u,v), \quad \textup{with} \quad H_\pm(u,v) \in \Lambda ^\pm SL(2,\C),
$$
and then setting $\hat F(u,v) = F_+(u) H_-(u,v)$.

Example solutions, using a numerical implementation of this method, are computed below.

\section{Limiting cases: pseudospherical surfaces in Euclidean space and flat surfaces in the $3$-sphere}  \label{limitingsection}
In this section we discuss the interpretation of admissible frames at degenerate values of the
loop parameter $\mu$, namely the case $\mu =1$, which was excluded from the above construction,
and the limit $\mu \to 0$ or $\mu \to \infty$.

\subsection{Relation to pseudospherical surfaces in Euclidean space ${\mathbb E}^3$.} 
As alluded to above, in addition to the constant Gauss curvature $K=1-\rho^2$ surfaces in 
$\SSS^3$ of Lemma \ref{converselemma}, one
also obtains, from a regular admissible frame $\hat F$, a constant negative curvature $-1$ surface
in $\mathbb{E}^3$ by the \emph{Sym formula}:
\beq  \label{symformula}
 \check f = 2 \frac{\partial \hat F}{\partial \lambda} \hat F^{-1} \big|_{\lambda=1}.  
\eeq
Here we explain how this formula arises naturally from the construction of surfaces in $\SSS^3= SU(2)$.
 
Obviously the projection formula 
\bdm
f_\mu = \hat F |_{\lambda=\mu} \, \hat F^{-1} |_{\lambda=1},
\edm for the surface in $\SSS^3$,
 degenerates to a constant map for $\mu = 1$.
On the other hand, we can see that 
$K_\mu = 1-(\mu+1)^2/(\mu-1)^2$ approaches $-\infty$ when $\mu$ approaches $1$.  This suggests that 
we multiply  our projection
formula by some factor, allowing the size of the sphere to vary,
 such that $K$ approaches some finite limit instead, in order to have an interpretation for the map at $\mu=1$.  Set
\beqas  
\check f_\mu & = & \frac{2}{1-\mu} (f_\mu -e_0)\\
& = & \frac{2} {1-\mu} \, ( \hat F \big|_{\lambda=\mu} \, \hat F^{-1} \big|_{\lambda = 1}-e_0). 
\eeqas
Note that $e_0=(1,0,0,0)$ under our identification $\mathbb{E}^4 = \mathfrak{su}(2) + \textup{span}(e_0)$.
Now, for $\mu \neq 1$, the function $f_\mu$ is a constant curvature $K_\mu$ surface
in $\SSS^3$, and $\check f_\mu$ is obtained by a constant dilation  of $\mathbb{E}^4$ by the 
factor $2(1-\mu)^{-1}$, plus a constant translation which has no geometric significance. 
It follows  that $\check f_\mu$ is a surface in a (translated) sphere of radius
$2(1-\mu)^{-1}$, and $\check f_\mu$ has constant curvature
\beqa
\check K_\mu &=& (1/4)(1-\mu)^2 K_\mu  \label{Khat} \\
  &=&  -\mu.   \nonumber
\eeqa

Now consider the function $g:  M \times (1-\varepsilon, 1+\varepsilon)  \to {\mathbb E}^4$, for some small
positive real number $\varepsilon$, given by
\bdm
g(u,v,\lambda) = 2 (\hat F(u,v)\big|_{\lambda=\lambda} \hat F(u,v)^{-1} \big|_{\lambda =1} -e_0).
\edm
This function is differentiable in all arguments, and 
\beqas
\left. \frac{ \partial g}{\partial \lambda} \right |_{\lambda=1} &=&  2
\lim_{\mu \to 1}  \frac{F \big|_{\lambda=\mu} F^{-1}\big|_{\lambda=1} - e_0}{1- \mu } \\
 &=& \lim_{\mu \to 1} \check f_\mu
\eeqas
Hence the limit on the right hand side exists and is a smooth function $M \to \mathbb{E}^4$.
On the other hand, differentiating the definition of $g$, we obtain the right hand side of
the Sym formula
(\ref{symformula}).
Note that, since $F$ is $SU(2)$-valued, this expression takes values in the Lie algebra,
$\mathfrak{su}(2) = \textup{span} (e_1, e_2, e_2)$,
 which, in our representation of $\mathbb{E}^4$, is the hyperplane
$x_0 = 0$.   In other words, $\lim_{\mu \to 1} \check f_\mu$ takes values in $\mathbb{E}^3 \subset \mathbb{E}^4$.  
Assuming that our surface in $\SSS^3$ is regular, then
one can verify that the regularity assumption on the frame $F$ implies that this map
is an immersion, and it is clear from the expression (\ref{Khat}) that this surface
has constant curvature $-1$.  

\begin{example}  \label{example1}
\begin{figure}[ht]
\centering
$
\begin{array}{cccc}
\includegraphics[height=40mm]{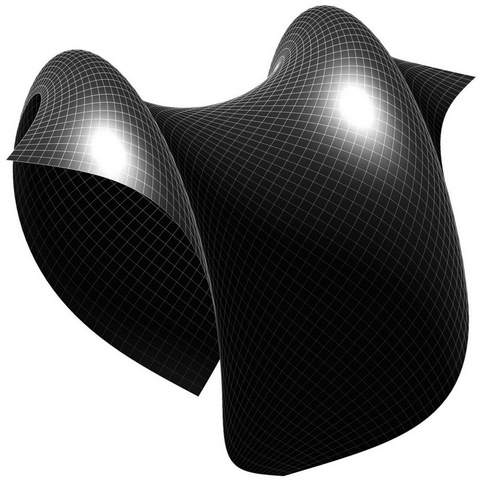}  \quad \quad & \quad \quad
\includegraphics[height=40mm]{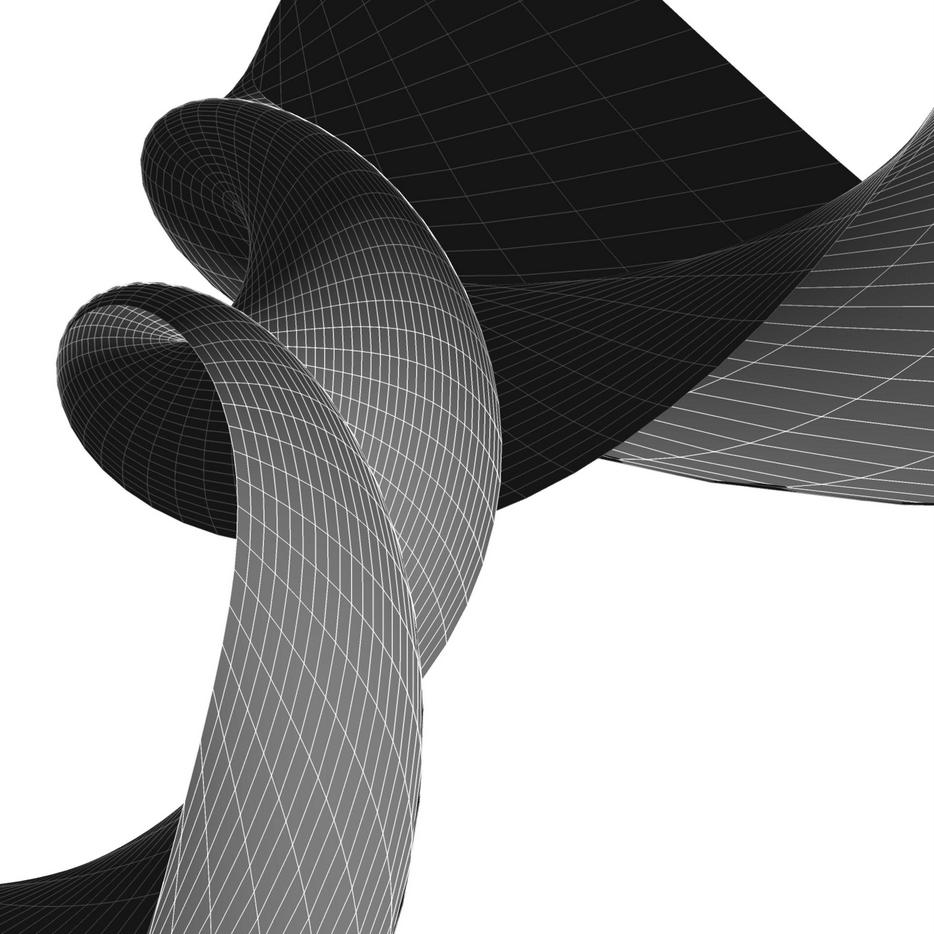}
 \vspace{1ex} \\
\mu=1, ~ K=-1, & 
 \mu=4, ~K=-\frac{16}{9},  \\
 \textup{target } {\mathbb E}^3. & \textup{target } \SSS^3.  \vspace{1ex} \\
\includegraphics[height=40mm]{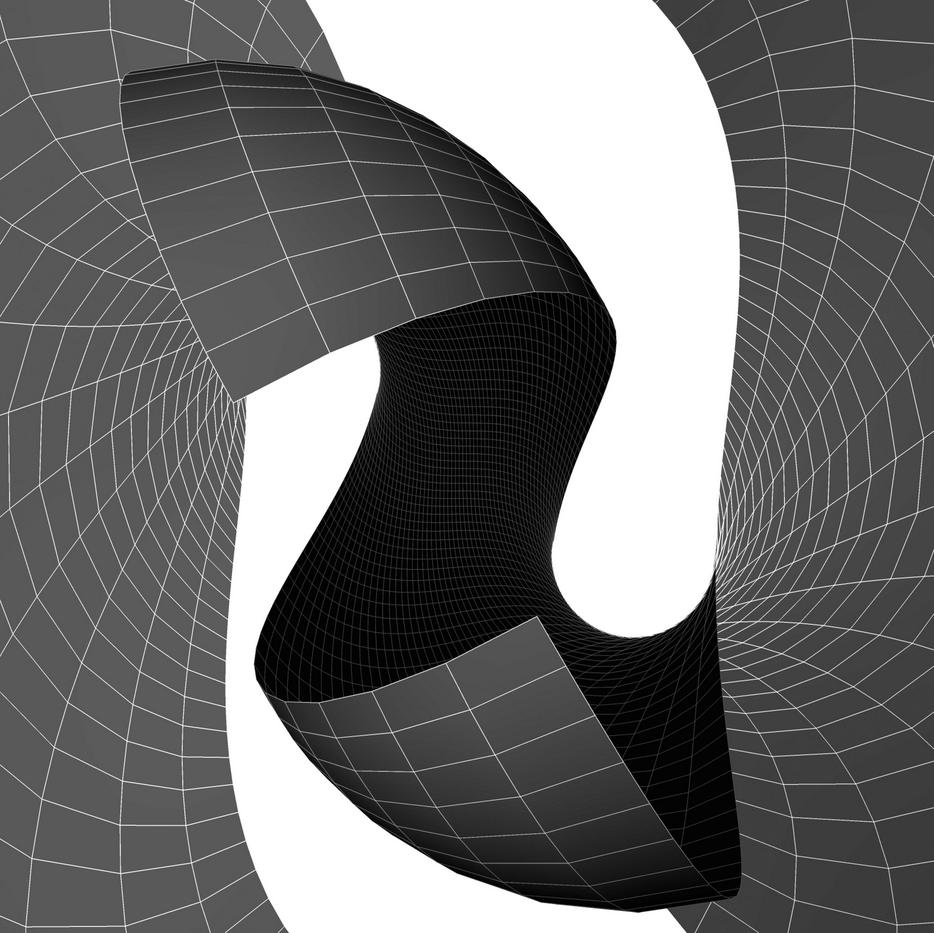}  \quad \quad & \quad \quad 
\includegraphics[height=40mm]{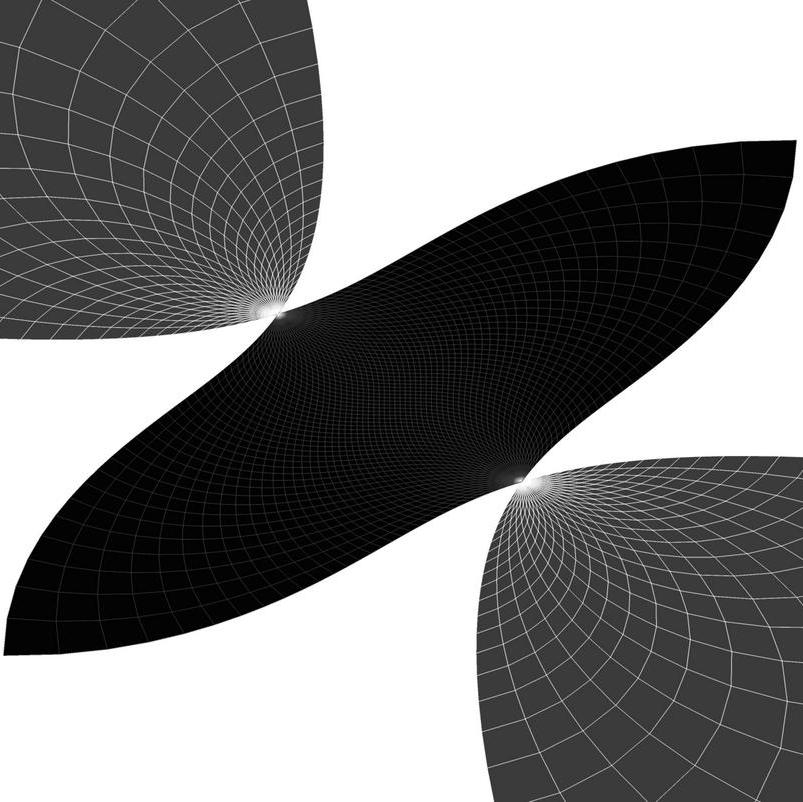}   \vspace{1ex} \\
 \mu=-4,~ K=\frac{16}{25},  & 
 \mu=-1,~ K=1, \\
 \textup{target } \SSS^3.&
 \textup{target } \SSS^3.
\end{array}
$
\caption{Surfaces obtained from one admissible frame evaluated at different values of $\mu$. All images are of the same coordinate patch. 
The first image is obtained via the Sym formula, and the others are in 
$\SSS^3$, stereographically projected to $\real^3$ for plotting. } 
\label{figure1}
\end{figure}
In Figure \ref{figure1}, 
various different projections of the same admissible frame are plotted.  
These are computed using the generalized d'Alembert method (see
\cite{todaagag}), using the potential pair
\bdm
\eta_+ = A \dd u, \quad \eta_- = A \dd v, \quad A = \bbar{cc} 0 & -\lambda^{-1} + i \lambda \\ \lambda^{-1} + i \lambda & 0 \ebar.
\edm 
The first image, the surface in $\mathbb{E}^3$ obtained via the Sym formula (\ref{symformula}),
is part of a hyperbolic surface of revolution (a plot of a larger region is shown in \ref{figureintro2}).
  The two cuspidal edges that can be seen in this image also appear in 
  the other surfaces at the same places in the coordinate domain, 
  because the condition on the admissible frame for the surface to be regular is independent of $\mu$.
The surfaces in $\SSS^3$ are of course distorted by the stereographic projection, 
which is taken from the south pole $(-1,0,0,0) \in \mathbb{E}^4$: the north pole,
$(1,0,0,0)$ is 
at the center of the coordinate domain plotted.  The last image is in fact planar, 
the projection of a part of a totally geodesic hypersphere $\SSS^2 \subset \SSS^3$.  
In this case, each of the two singular curves in the coordinate domain maps to a single point in 
the surface.
\end{example}

\begin{example}  \label{example2}
Amsler's surface in $\mathbb{E}^3$ can be computed by the generalized d'Alembert method 
using the potential pair:
\bdm
\eta_+ = \bbar{cc} 0 & i \lambda \\ i \lambda & 0 \ebar \dd u,
 \quad \eta_- = \bbar{cc} 0 & -\lambda^{-1} \\ \lambda^{-1} & 0 \ebar \dd v.
\edm 
The image of a rectangle $[0, a] \times [0,b]$ in the positive quadrant of 
the $uv$-plane is plotted in Figure \ref{figureA}, evaluated at 3 different values of $\mu$.  
The coordinate axes correspond to straight lines for the surface in $\mathbb{E}^3$, and 
to great circles for the surfaces in $\SSS^3$, which project to straight lines 
 under the stereographic projection from the south pole.  The  north pole $(1,0,0,0)$ 
corresponds to $(u,v)=(0,0)$.

The singular set in the coordinate patch corresponds to a cuspidal edge in each of the first two images, but contains a swallowtail singularity in the third.
See also Figure \ref{figureintro1}.
\begin{figure}[ht]
\centering
$
\begin{array}{ccc}
\includegraphics[height=34mm]{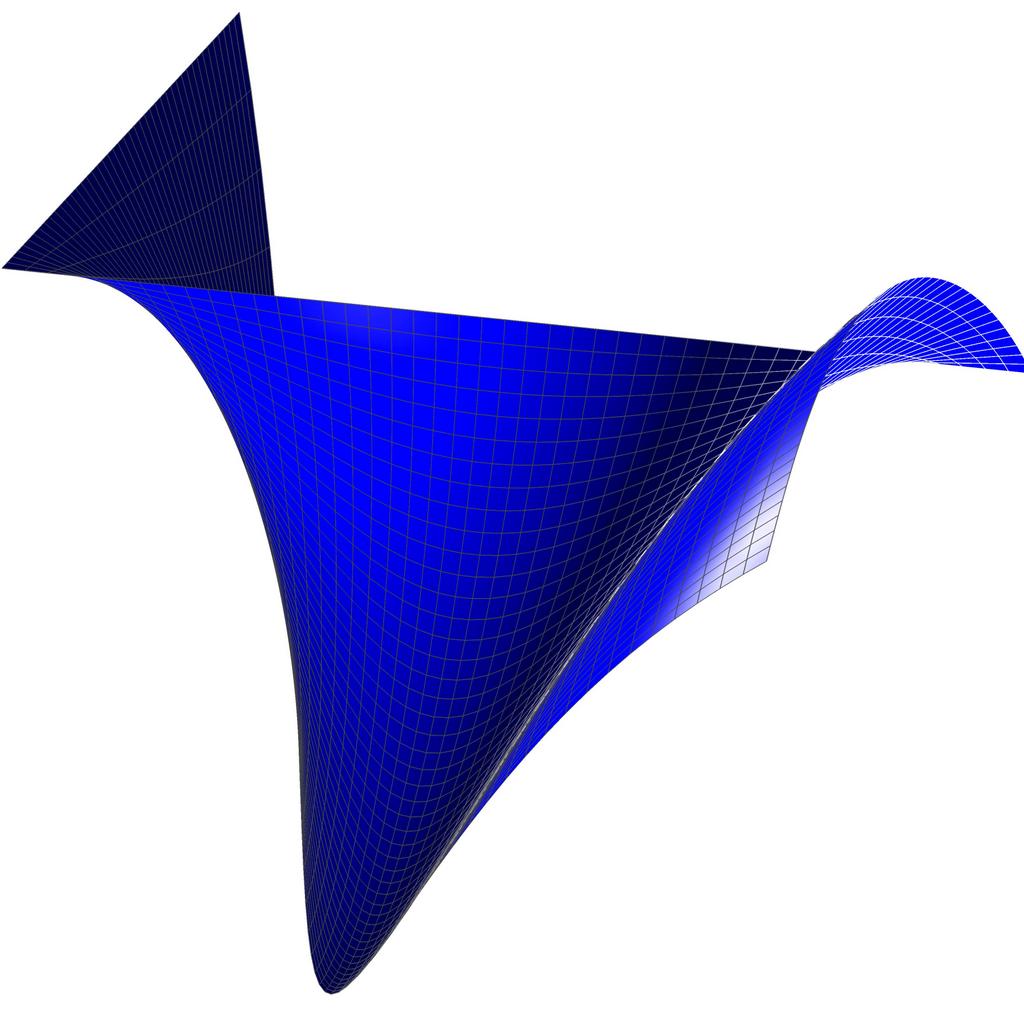}  \quad & \quad
\includegraphics[height=34mm]{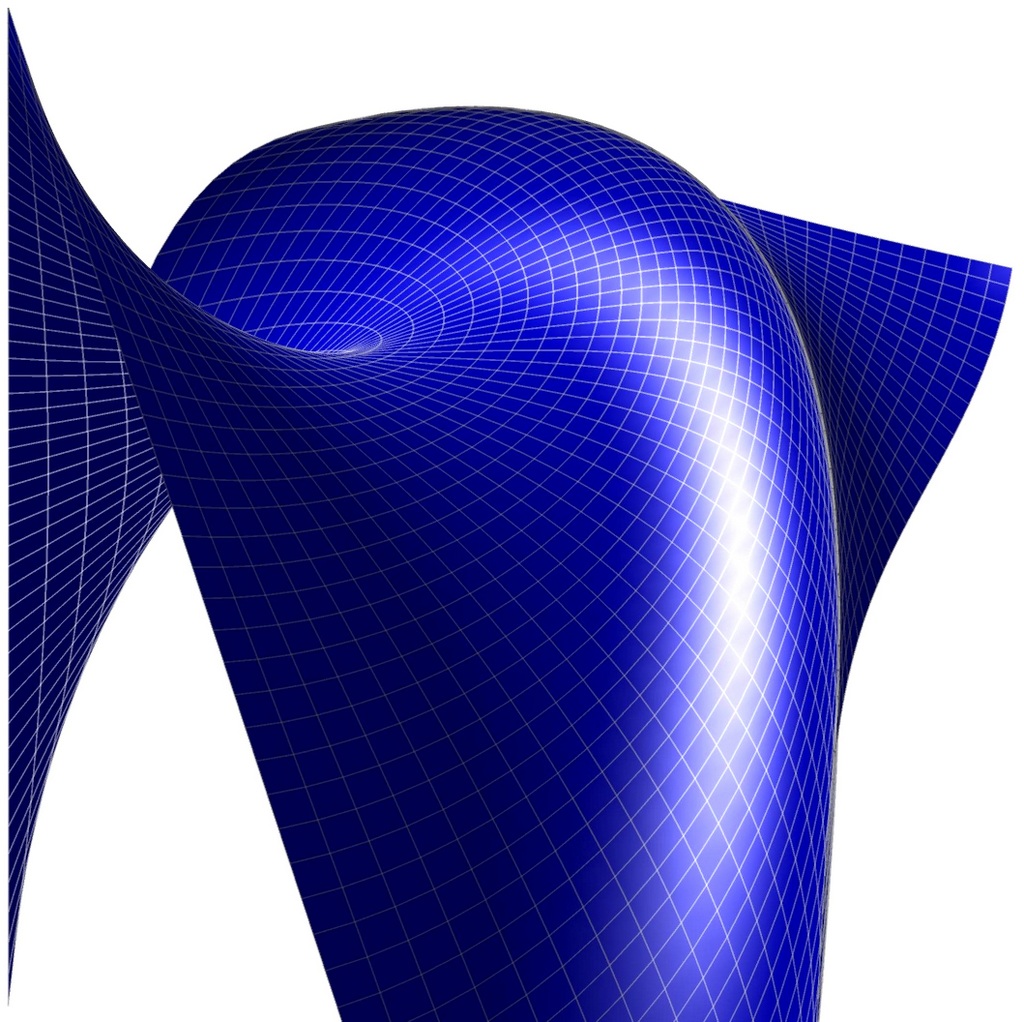}   \quad & \quad
\includegraphics[height=34mm]{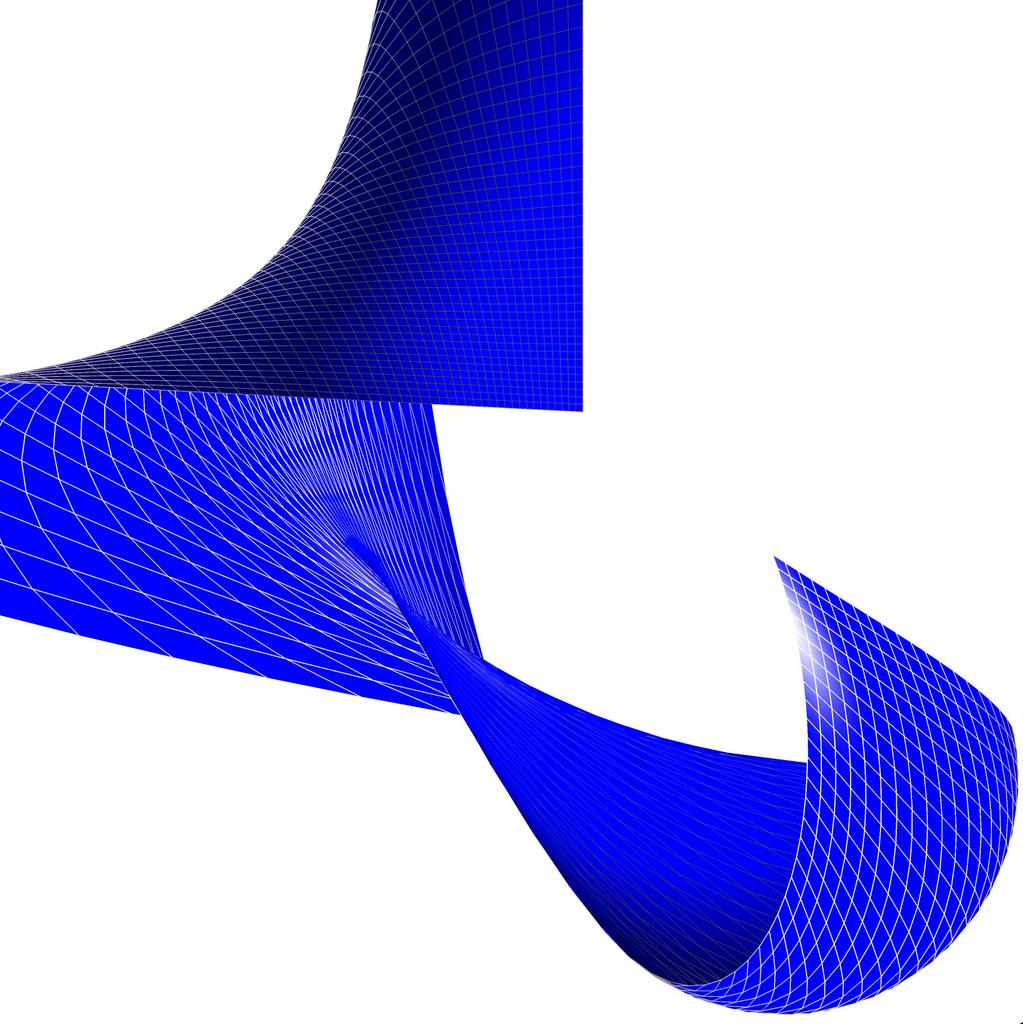}  
 \vspace{2ex} \\
\mu=1, ~ K=-1, & 
 \mu=4, ~K=-\frac{16}{9},  &  \mu=-4,~ K=\frac{16}{25},  \\
 \textup{target } {\mathbb E}^3. & \textup{target } \SSS^3. & \textup{target } \SSS^3.
\end{array}
$
\caption{Amsler's surface and generalizations in the $3$-sphere. The surfaces are obtained from one admissible frame evaluated at different values of $\mu$. All images are of the same coordinate patch.  } 
\label{figureA}
\end{figure}
\end{example}
\subsection{Relation to flat surfaces in the $3$-sphere}
We have considered above the surfaces $f_\mu$, obtained   by the projection 
\beq  \label{projformula3}
\hat F |_{\lambda=\mu} \, \hat F^{-1} |_{\lambda=1},
\eeq
 for all non-zero real values of $\mu$. We now consider the limit as $\mu$ approaches $0$ or $\infty$.
From the formula $K_\mu = 1-(\mu+1)^2/(\mu-1)^2$, it is clear that the limiting surface, if it exists, will be flat.  We discuss the case $\mu \to 0$ here.

Observe that the admissible frame $\hat F$ has a pole at $\lambda =0$, so we cannot
evaluate (\ref{projformula3}) at $\mu=0$.  However, in the Maurer-Cartan form of $\hat F$, the factor $\lambda^{-1}$ appears only as a coefficient of $\dd v$.  Hence a change of coordinates could remove the pole in $\lambda$. For $\mu >0$, we set
$\tilde u = u$ and $\tilde v = v/\mu$, so that
\bdm
 f_\mu ( u,  v) = f_\mu(\tilde u, \mu \tilde v) =: g_\mu(\tilde u, \tilde v).
\edm
For simplicity, let us assume that $M$ is a rectangle $(a,b) \times (c,d) \subset \real^2$, containing the origin $(0,0)$ and with coordinates $(u,v)$.  We denote by
$M_\mu$ the same rectangle in the coordinates $(\tilde u, \tilde v)$, that is
$M_\mu = (a,b) \times (c/\mu, d/\mu)$, and we define $M_0 := (a,b) \times (-\infty, \infty)$. 
 
We have already seen that, for $\mu>0$,
 the map $g_\mu: M_\mu \to \SSS^3$ is an immersion
of constant curvature $K_\mu = 1-(\mu+1)^2/(\mu-1)^2$, since this  is just the same map as 
$f_\mu$ in different coordinates.  For fixed $\mu_0 \in (0,1)$, 
if $0<\mu<\mu_0$ then $M_{\mu} \supset M_{\mu_0}$, and so we can restrict $g_\mu$ to $M_{\mu_0}$ and talk about a family of maps $g_\mu: M_{\mu_0} \to \SSS^3$ with
a fixed domain. 
\begin{figure}[ht]
\centering
\includegraphics[height=50mm]{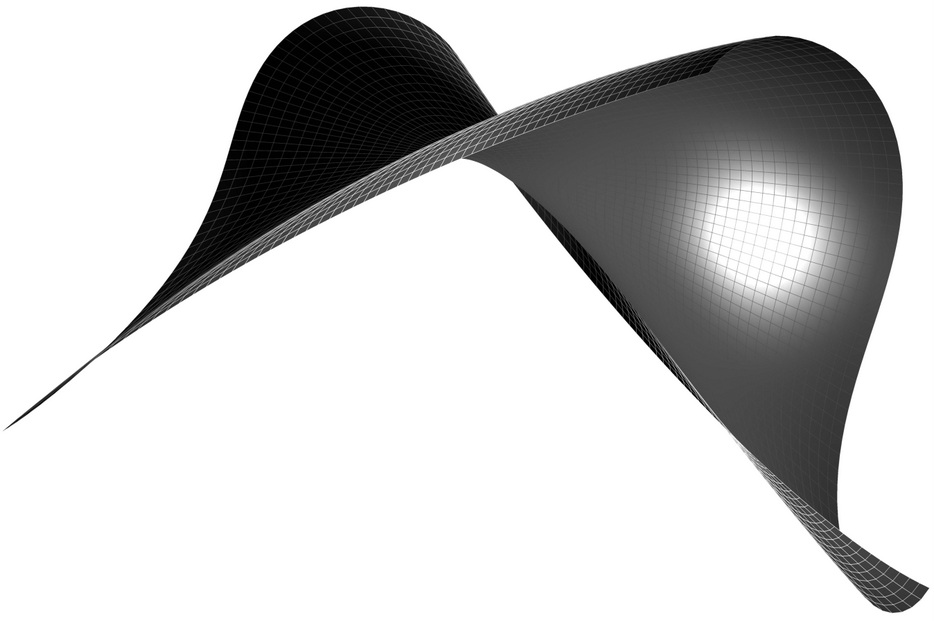}  
\caption{The surface $g_\mu$, for $\mu= 10^{-9}$, obtained from the same admissible frame used in Figure \ref{figure1}. } 
\label{figure2}
\end{figure}
%
%
%
\begin{lemma}
For any fixed $\mu_0 \in (0,1)$, the family of maps $g_\mu: M_{\mu_0} \to \SSS^3$ extends real analytically in $\mu$ to $\mu =0$.  Moreover, the map $g_0: M_{\mu_0} \to \SSS^3$ extends to the whole of $M_0 = (a,b) \times (-\infty,\infty)$, and is an
immersion of zero Gaussian curvature.
\end{lemma}

\begin{proof}
Write $\hat G_\mu (\tilde u, \tilde v) = \hat F(\tilde u, \mu \tilde v) = \hat F(u,v)$, so that $\hat G_\mu : M_\mu \to \mathcal{G}$.   Then 
\bdm
g_\mu(\tilde u, \tilde v) = H_\mu (\tilde u, \tilde v) K_\mu^{-1}(\tilde u, \tilde v),
\quad \textup{where} \quad
H_\mu  := \hat G_\mu  \big |_{\lambda = \mu}, \quad
K_\mu := \hat G_\mu \big |_{\lambda=1}.
\edm
Since $\hat F$ is an admissible frame, we can write
\beqas
\hat F^{-1} \dd \hat F &=& (U_0 + \lambda U_1) \dd u + (V_0 + \lambda^{-1} V_1) \dd v \\
& =& (U_0(\tilde u, \mu \tilde v) + \lambda U_1(\tilde u, \mu \tilde v))
\dd \tilde u + (\mu V_0 (\tilde u, \mu \tilde v) +  \mu\lambda^{-1} V_1 (\tilde u, \mu \tilde v) )\dd \tilde v, \eeqas
 and thus
\bdm
H_\mu^{-1} \dd H_\mu = (U_0 + \mu U_1) \dd \tilde u + (\mu V_0 + V_1) \dd \tilde v,
\edm
so 
\bdm
H_0^{-1} \dd H_0 = U_0(\tilde u, 0)  \dd \tilde u +  V_1 (\tilde u,0) \dd \tilde v,
\edm
and 
\bdm
K_\mu^{-1} \dd K_\mu = (U_0 +  U_1) \dd \tilde u + (\mu V_0 +  \mu V_1) \dd \tilde v,
\edm
so
\bdm
K_0^{-1} \dd K_0 = (U_0(\tilde u, 0) +  U_1(\tilde u, 0)) \dd \tilde u.
\edm
Since $H_\mu$ and $K_\mu$ are both obviously real analytic in 
$\mu$ in a neighbourhood of $\mu=0$, so also is $g_\mu$. Finally the 1-forms 
$\gamma = H_0^{-1} \dd H_0$ and $\delta = K_0^{-1} \dd {K_0}$ are both integrable
on $M_{\mu_0} =(a,b) \times (c/\mu_0, d/\mu_0)$ for any fixed $\mu_0$. But, since
the coefficients of the $1$-forms are constant in $\tilde v$, this means that they
are in fact integrable on the whole of $(a,b) \times (-\infty,\infty)$. This implies the claim.
\end{proof}
 
Using the expressions $\gamma$ and $\delta$ above, we obtain the formula
\bdm
g_0^{-1} \dd g_0 = \Ad_{K_0}(-U_1(\tilde u,0) \dd \tilde u + V_1 (\tilde u, 0) \dd \tilde v),
\edm
from which we have the following expression for the first  fundamental form of $g_0$:
\bdm
I(\tilde u, \tilde v) = \left( |B_1|^2 \dd \tilde u ^2 -2 \cos(\phi) |B_1| |B_{-1}| \dd \tilde u \dd \tilde v + |B_{-1}|^2 \dd \tilde v ^2 \right) \, \Big|_{(\tilde u, 0)}.
\edm
Letting $\mu \to 0$ in the expression (\ref{firstandsecondff2}), we conclude that
the second fundamental form of $g_0$ is
\bdm
\mathrm{II} = 2|B_1(\tilde u,0)| |B_{-1}(\tilde u,0)| \sin (\phi(\tilde u,0)) \dd \tilde u \dd \tilde v.
\edm

\begin{example} 
In Figure \ref{figure2} 
is shown the surface $g_\mu$, for $\mu=10^{-9}$, obtained from the 
same admissible frame $\hat F_\mu$ used in Example \ref{example1}.  A square region in the
$(\tilde u, \tilde v)$-plane is plotted, approximately equal to the region $(a,b) \times (c/\mu, d/\mu)$ in the $uv$-plane, where the region plotted in Example \ref{example1} was
$(a,b) \times (c,d)$.  The region plotted here is actually slightly larger, in order to make
the singular set visible.  As $\mu$ approaches zero, the cuspidal edges, which, in the
non-flat surface, were something of the form $v=\pm u + \textup{constant}$, are now approaching curves of the form
$\tilde v= \textup{constant}$.
\end{example}


\end{document}